\documentclass{article}

\usepackage{amsmath,amsfonts,amssymb,amsthm, tikz}
\usepackage{xcolor}
\usepackage{url}
\usepackage[all]{xy}
\usepackage{hyperref}
\usepackage[capitalize]{cleveref}

\theoremstyle{definition}
\newtheorem{definition}{Definition}[section]
\newtheorem{remark}[definition]{Remark}
\newtheorem{example}[definition]{Example}

\theoremstyle{plain}
\newtheorem{thm}[definition]{Theorem}
\newtheorem{proposition}[definition]{Proposition}
\newtheorem{lemma}[definition]{Lemma}
\newtheorem{corollary}[definition]{Corollary}

\newcommand{\trop}{\operatorname{trop}}
\newcommand{\Sol}{\operatorname{Sol}}
\renewcommand{\Vert}{\operatorname{Vert}}
\newcommand{\Supp}{\operatorname{Supp}}
\newcommand{\Val}{\operatorname{Val}}
\title{The Fundamental Theorem of Tropical Partial Differential Algebraic Geometry}

\author{Sebastian Falkensteiner\and Cristhian Garay-L\'opez\and Mercedes Haiech\and Marc Paul Noordman\and Zeinab Toghani \and Fran{\c{c}}ois Boulier}

\begin{document}

\maketitle
\let\thefootnote\relax\footnotetext{
\emph{2010 Mathematics Subject Classification}. 13P15, 13N99, 14T99, 52B20.\\
\emph{Keywords and phrases.} Differential Algebra, Tropical Differential Algebraic Geometry, Power Series Solutions, Newton Polygon, Arc Spaces.

This research project and the fifth author was supported by the European Commission, having received funding from the European Union's Horizon 2020 research and innovation programme under grant agreement number 792432. The first author was supported by the Austrian Science Fund (FWF): P 31327-N32. The second author was supported by CONACYT through Project 299261. 
The sixth author would like to thank the bilateral project ANR-17-CE40-0036 and DFG-391322026 SYMBIONT for its support.}
\let\thefootnote\svthefootnote

\begin{abstract}
Tropical Differential Algebraic Geometry considers difficult or even intractable problems in Differential Equations and tries to extract information on their solutions from a restricted structure of the input. 
The Fundamental Theorem of Tropical Differential Algebraic Geometry states that the support of solutions of systems of ordinary differential equations with formal power series coefficients over an uncountable algebraically closed field of characteristic zero can be obtained by solving a so-called tropicalized differential system. Tropicalized differential equations work on a completely different algebraic structure which may help in theoretical and computational questions. We show that the Fundamental Theorem can be extended to the case of systems of partial differential equations by introducing vertex sets of Newton polygons.
\end{abstract}

\section{Introduction}
Given an algebraically closed field of characteristic zero $K$, we consider the partial differential ring $(R_{m,n},D)$, where $$R_{m,n}=K[[t_1,\ldots,t_m]]\{x_1,\ldots,x_n\}$$ and $D=(\tfrac{\partial}{\partial t_k}\::\:k=1,\ldots,m)$ for $n,m\geq1$ (see Section \ref{S_PDAG} for definitions).
Up to now, tropical differential algebra has been limited to the study of the relation between the set of solutions $\Sol(G)\subseteq K[[t]]^n$ of  differential ideals $G$ in $R_{1,n}$ and their corresponding {\it tropicalizations}, which are certain 
polynomials $p$ with coefficients in a tropical semiring $\mathbb{T}[t]:=(\mathbb{Z}_{\geq0}\cup\{\infty\},+,\min{})$ and with a set of solutions $\Sol(p) \subseteq \mathcal{P}(\mathbb{Z}_{\geq0})^n$, see~\cite{grigoriev2017tropical} and~\cite{AGT16}. These elements $S\in \Sol(p)$ can be found by looking at {\it evaluations} $p(S)\in \mathbb{T}[t]$ where the usual tropical vanishing condition holds.

\medskip
In this paper, we consider the case $m>1$. On this account, we work with elements in $\mathbb{Z}_{\ge 0}^m$, which requires new techniques. We show that considering the Newton polygons and their vertex sets is the appropriate method for formulating and proving our generalization of the Fundamental Theorem of Tropical Differential Algebraic Geometry. 
We remark that in the case of $m=1$ the definitions and properties presented here coincide with the corresponding ones in~\cite{AGT16} and therefore, this work can indeed be seen as a generalization.

\medskip
The problem of finding power series solutions of systems of partial differential equations has been extensively studied in the literature, but is very limited in the general case. In fact, we know from~\cite[Theorem 4.11]{DenefLipshitz} that there is already no algorithm for deciding whether a given linear partial differential equation with polynomial coefficients has a solution or not. The Fundamental Theorem, as it is stated in here, helps to find necessary conditions for the support of possible solutions.

\medskip
The structure of the paper is as follows. In Section~\ref{S_PDAG} we cover the necessary material from partial differential algebra. In Section~\ref{sec:the tropical semi-ring} we introduce the semiring of supports $\mathcal{P}(\mathbb{Z}_{\geq0}^m)$, the semiring of vertex sets $\mathbb{T}[[t_1,\ldots,t_m]]$ and the vertex homomorphism $\Vert\colon\mathcal{P}(\mathbb{Z}_{\geq0}^m)\longrightarrow\mathbb{T}[[t_1,\ldots,t_m]]$. In Section~\ref{S_DRPS} we introduce the support and the tropicalization maps. In Section~\ref{S_TDP} we define the set of tropical differential polynomials $\mathbb{T}_{m,n}$, the notion of tropical solutions for them, and the tropicalization morphism $\trop\colon R_{m,n}\to\mathbb{T}_{m,n}$. The main result is Theorem~\ref{fundamental theorem}, which is proven in Section~\ref{S_TFT}. The proof we give here differs essentially from that one in~\cite{AGT16} for the case of $m=1$. In Section~\ref{Section_ExLim} we give some examples to illustrate our results.

\medskip
In the following we will use the conventions that for a set $S$ we denote by $\mathcal{P}(S)$ its power set, and by $K$ we denote an algebraically closed field of characteristic zero.

\section{Partial differential algebraic geometry}
\label{S_PDAG}
Here we recall the preliminaries for partial differential algebraic geometry. The reference book for differential algebra is~\cite{Kolchin73}.

\medskip
A \textbf{partial differential ring} is a pair $(R,D)$ consisting of a commutative ring with unit $R$ and a set $D=\{\delta_1,\ldots,\delta_m\}$ of $m > 1$  \textbf{derivations} which act on~$R$ and are pairwise commutative. We denote by $\Theta$ the free commutative monoid generated by $D$. If $J=(j_1,\ldots,j_m)$ is an element of the monoid  $\mathbb{Z}^m_{\geq0}=(\mathbb{Z}^m_{\geq0},+,0)$,
we denote $\Theta(J) = \delta_1^{j_1} \cdots \delta_m^{j_m}$
the \textbf{derivative operator} defined by~$J$.
If $\varphi$ is any element of~$R$, then $\Theta(J)\varphi$ is
the element of~$R$ obtained by application of the derivative operator $\Theta(J)$ on~$\varphi$.

\medskip
Let $(R,D)$ be a partial differential ring and $x_1,\ldots,x_n$
be~$n$ \textbf{differential indeterminates}.
The monoid~$\Theta$
acts on the differential indeterminates, giving 
the infinite set of the \textbf{derivatives} which are denoted by
$x_{i,J}$ with $1 \leq i \leq n$ and $J \in \mathbb{Z}^m_{\geq0}$.
Given any $1 \leq k \leq m$ and any derivative $x_{i,J}$, the action of~$\delta_k$ on $x_{i,J}$ is
defined by $\delta_k(x_{i,J})=x_{i,J+e_k}$ where $e_k$ is the $m$-dimensional vector whose $k$-th coordinate is~$1$ and 
all other coordinates are zero. 
One denotes $R\{x_1,\ldots,x_n\}$ the ring of the polynomials, with coefficients in $R$, the indeterminates of which are the derivatives. More formally, $R\{x_1,\ldots, x_n\}$ consists of all $R$-linear combinations of differential monomials, where a differential monomial in $n$ independent variables of order less than or equal to $r$ is an expression of the form

\begin{equation}
\label{differential_monomial}
    E_M:=\prod_{\substack{1\leq i\leq n\\ ||J||_\infty\leq r}}x_{i,J}^{M_{i,J}} 
\end{equation}
where $J=(j_1,\ldots,j_m)\in \mathbb{Z}_{\geq0}^m$, $||J||_\infty:=\text{max}_i\{j_i\}=\text{max}(J)$ and $M=(M_{i,J}) \in (\mathbb{Z}_{\ge0})^{n\times(r+1)^m}$. 

\medskip
The pair $(R\{x_1,\ldots,x_n\},D)$ then constitutes a \textbf{differential polynomial ring}. 
A differential polynomial  $P\in R\{x_1,\ldots,x_n\}$ induces an evaluation map from $R^n$ to $R$ given by
\[P\colon R^n \to R,\quad (\varphi_1,\ldots,\varphi_n)\mapsto P|_{x_{i,J}=\Theta(J)\varphi_i},\]
where $P|_{x_{i,J}=\Theta(J)\varphi_i}$ is the element of $R$ obtained by substituting $\Theta(J)\varphi_i$ for $x_{i,J}$ .
 
\medskip
A \textbf{zero} or \textbf{solution} of $P\in R\{x_1,\ldots,x_n\}$ is an $n$-tuple $\varphi=(\varphi_1,\ldots,\varphi_n)\in R^n$ such that $P(\varphi)=0$.
An $n$-tuple $\varphi \in R^n$ is a solution of a system of differential polynomials~$\Sigma \subseteq R\{x_1,\ldots,x_n\}$ if it is a solution of every element of~$\Sigma$. We denote by $\Sol(\Sigma)$ the solution set of the system $\Sigma$.

\medskip
A \textbf{differential ideal} of $R\{x_1,\ldots,x_n\}$
is an ideal of that ring which is stable under the
action of $\Theta$. A differential ideal
is said to be \textbf{perfect} if it is equal to
its radical. If $\Sigma \subseteq R\{x_1,\ldots,x_n\}$, one denotes by $[\Sigma]$ the \textbf{differential ideal generated by~$\Sigma$} and by $\{ \Sigma \}$ the \textbf{perfect differential ideal generated by~$\Sigma$}, which is defined as the intersection of all perfect
differential ideals containing~$\Sigma$.

\medskip
For $m,n\geq1$, we will denote by $R_m$ the partial differential ring 
$$(K[[t_1,\ldots,t_m]],D)$$ 
where $D=\{\tfrac{\partial}{\partial t_1},\ldots,\tfrac{\partial}{\partial t_m}\}$, and the partial differential ring $(R_m\{x_1,\ldots,x_n\},D)$ will be denoted by $R_{m,n}$ . The proof of the following proposition can be found in ~\cite{BH19}.
 
\begin{proposition}\label{prop_Ritt_Raudenbush}
For any $\Sigma \subseteq R_{m,n}$, there exists a finite subset~$\Phi$ of~$\Sigma$ such that $\Sol(\Sigma)=\Sol(\Phi)$.
\end{proposition}

\section{The semirings of supports and vertex sets}
\label{sec:the tropical semi-ring}
In this part we introduce and give some properties on our main idempotent semirings, namely the semiring of supports $\mathcal{P}(\mathbb{Z}_{\geq0}^m)$, the semiring of vertex sets $\mathbb{T}[[t_1,\ldots,t_m]]$ and the map $\Vert\colon\mathcal{P}(\mathbb{Z}_{\geq0}^m)\rightarrow\mathbb{T}[[t_1,\ldots,t_m]]$ which is a homomorphism of semirings.

\medskip
Recall that a commutative semiring $S$ is a tuple $(S,+,\times,0,1)$ such that $(S,+,0)$ and $(S,\times,1)$ are commutative monoids and additionally, for all $a,b,c \in S$ it holds that
\begin{enumerate}
    \item $a\times (b+c)=a\times b+a\times c$;
    \item $0\times a=0$.
\end{enumerate} 
A semiring is called \textbf{idempotent} if $a + a = a$ for all $a \in S$. 
A map $f\colon S_1\longrightarrow S_2$ between semirings is a morphism if it induces morphisms at the level of monoids.

\medskip
For $m\geq1$, we denote by $\mathcal{P}(\mathbb{Z}^m_{\geq0})$ the idempotent semiring whose elements are the subsets of $\mathbb{Z}^m_{\geq0}$ equipped with the union $X\cup Y$ as sum and the Minkowski sum $X+Y=\{x+y\::\:x\in X, y\in Y\}$ as product. 
We call it the \textbf{semiring of supports}.
For $n\in\mathbb{Z}_{\geq1}$ and $X \in \mathcal{P}(\mathbb{Z}_{\geq 0}^m)$, 
the notation $nX$ will indicate $\underbrace{X+ \cdots + X}_{n \ \text{times}}$. By convention we set $0X = \{(0,\ldots,0)\}$.

\medskip
We define the \textbf{Newton polygon} $\mathcal{N}(X)\subseteq\mathbb{R}^m_{\geq0}$ of $X\in\mathcal{P}(\mathbb{Z}^m_{\geq0})$ as the convex hull of $X+\mathbb{Z}^m_{\geq0}$. We call $x \in X$ a \textbf{vertex} if $x \notin \mathcal{N}(X \setminus \{x\})$, and we denote by $\Vert X$ the set of vertices of $X$.

\begin{lemma}\label{N_equal_impies_Vert_equal}
Let $S,T \in \mathcal{P}(\mathbb{Z}_{\geq0}^m)$ such that $\mathcal{N}(S)= \mathcal{N}(T)$. Then $\Vert S=\Vert T$.
\end{lemma}
\begin{proof}

Let $s \in \Vert S$ and we assume that $s \in \mathcal{N}(T \setminus\{s\})$. Then there are $t_i \in T \setminus \{s\}$, $w_i \in \mathbb{Z}_{\geq 0}^m$ and positive $\lambda_i \in \mathbb{R}$ adding up to 1 such that 
\[ s = \sum_{i} \lambda_i (t_i + w_i).\]
Since $t_i \in \mathcal{N}(S)$, we can write the $t_i$ as
\[ t_i = \sum_{j}\mu_{i,j} (s_{i,j} + z_{i,j}),\]
where $s_{i,j} \in S$, $z_{i,j} \in \mathbb{Z}_{\ge 0}^m$ and $\mu_{i,j} \in \mathbb{R}$ are positive and adding up to 1. Thus, 
\[ s = \sum_{i,j} \lambda_{i}\mu_{i,j} (s_{i,j} + z_{i,j} + w_i) = \sum_{i,j}\lambda_i \mu_{i,j} s_{i,j} + v,\]
where $v$ is a vector with non-negative coefficients. By excluding in the sum those summands $s_{i,j}$ which are equal to $s$, we obtain
\[ s = cs + \sum_{\substack{i,j\\s_{i,j} \neq s}} \lambda_i\mu_{i,j}s_{i,j} + v\]
where $c = \sum_{i,j: s_{i,j = s}} \lambda_i \mu_{i,j} \in [0,1]$. If $c < 1$ we can solve the equation above for $s$ to get 
\[ s = \sum_{\substack{i,j\\s_{i,j} \neq s}} \frac{\lambda_i\mu_{i,j}}{1 - c} s_{i,j} + \frac{v}{1-c}.\]
The coefficients for the $s_{i,j}$ are positive and sum to 1, so the summation in the right hand side gives an element of $\mathcal{N}(S \setminus\{s\})$. Since $\mathcal{N}(S\setminus\{s\})$ is closed under adding elements of $\mathbb{R}_{\geq 0}^m$, and the coefficients of $v/(1-c)$ are non-negative, we then find that $s \in \mathcal{N}(S \setminus\{s\})$ in contradicting to the assumption that $s$ is a vertex of $S$. If $c = 1$, then all $s_{i,j}$ are equal to $s$ and we get $s = s + v$. Therefore, $v = 0$ and $t_i = s$ for each $i$, and in particular $s \in T \setminus \{s\}$, which is a contradiction. So we conclude that $s \notin \mathcal{N}(T \setminus \{s\})$ and $s$ is a vertex of $T$.
\end{proof}

\begin{lemma}
Let $X \in \mathcal{P}(\mathbb{Z}_{\geq0}^m)$. Then $\mathcal{N}(\Vert X) = \mathcal{N}(X)$. 
\end{lemma}
\begin{proof}
By Dickson's lemma~\cite[chap. 2, Thm 5]{CLS92}, there is a finite subset $S \subseteq X$ with $X \subseteq S + \mathbb{Z}_{\geq 0}^m$. For such $S$, it holds that $\mathcal{N}(X) = \mathcal{N}(S)$ and by \cref{N_equal_impies_Vert_equal}, we get $\Vert X = \Vert S$. Therefore, replacing $X$ by $S$, we may assume that $X$ is finite. 

\medskip
We proceed by induction on $\#X$. Indeed, if $X = \emptyset$, the statement is obvious. Let $X$ be an arbitrary finite set. If every element of $X$ is a vertex of $X$, then $\mathcal{N}(X) = \mathcal{N}(\Vert X)$ is trivially true. Else, take $x \in X \setminus \Vert X$ and let $Y = X \setminus\{x\}$. Then $\mathcal{N}(X) = \mathcal{N}(Y)$ by definition, so applying \cref{N_equal_impies_Vert_equal} again we obtain $\Vert X = \Vert Y$. Since $\# Y < \#X$, we may apply the induction hypothesis to $Y$, and get that $\mathcal{N}(X) = \mathcal{N}(Y) = \mathcal{N}(\Vert Y) = \mathcal{N}(\Vert X)$.
\end{proof}

\begin{corollary}\label{equal_Vert_is_equal_N}
For $X, Y \in \mathcal{P}(\mathbb{Z}_{\geq0}^m)$ we have $\Vert X = \Vert Y$ if and only if $\mathcal{N}(X) = \mathcal{N}(Y)$. 
\end{corollary}

\begin{lemma}
\label{lemma_Properties_vert}
For $X,Y\in\mathcal{P}(\mathbb{Z}_{\geq0}^m)$, we have
\[\begin{array}{ccl}
\Vert(\Vert(X)\cup\Vert(Y)) & = &  \Vert(\Vert(X)\cup Y) \\
     & =& \Vert(X\cup\Vert(Y)) \\
     &=& \Vert(X\cup Y)
\end{array}\] 
and 
\[\begin{array}{ccl}
\Vert(\Vert(X)+\Vert(Y)) & = & \Vert(\Vert(X)+ Y) \\
     & =& \Vert(X+\Vert(Y)) \\
     &=& \Vert(X+ Y).
\end{array}\]
\end{lemma}

\begin{proof}
Let $*$ be either $\cup$ or $+$. We have the following  diagram of inclusions
\begin{equation*}
\xymatrix{
&\Vert(X)* Y\ar[dr]&\\
\Vert(X)*\Vert(Y)\ar[dr]\ar[ur]\ar[rr]&&X*Y\\
&X* \Vert(Y)\ar[ur]&\\
}
\end{equation*}
We show that these four sets generate the same Newton polygon. For this, it is enough to show that $X*Y\subseteq \mathcal{N}(\Vert(X)*\Vert(Y))$.

\medskip
For $* = \cup$, we have $X \subseteq \mathcal{N}(\Vert X) \subseteq \mathcal{N}(\Vert(X) \cup \Vert(Y))$ and similarly $Y \subseteq \mathcal{N}(\Vert(X) \cup \Vert(Y))$. Hence, $X \cup Y \subseteq \mathcal{N}(\Vert(X) \cup \Vert(Y))$.

\medskip
Now suppose that $*=+$. Let $t\in X+Y$, and write $t = x + y$ with $x \in X$ and $y \in Y$. Using the inclusions $X \subseteq \mathcal{N}(\Vert X)$ and $Y \subseteq \mathcal{N}(\Vert Y)$, there are $x_{i}\in\Vert(X)$, $y_{j}\in \Vert(Y)$, $u_{i}, v_{j}\in \mathbb{Z}_{\geq0}^m$ and $\alpha_{i},\beta_{j} \in \mathbb{R}_{\ge 0}$ satisfying $\sum_{i}\alpha_{i} = 1$ and $\sum_{j}\beta_{j} = 1$ such that
\[t=\sum_i \alpha_{i}(x_{i} + u_{i})+\sum_j \beta_{j}(y_{j} + v_{j}).\] Rewriting this gives
\[t = \sum_{i,j} \alpha_{i}\beta_{j}(x_{i} + y_{j} + u_{i} + v_{j}).\]
For each pair $i,j$, the expression between parentheses is an element of $\Vert(X) + \Vert(Y) + \mathbb{Z}_{\geq 0}$ and the coefficients are non-negative and sum up to 1. This shows that $t \in \mathcal{N}(\Vert(X) + \Vert(Y))$, which ends the proof of the inclusions.
\end{proof}

\begin{example}
\label{Example_VertSet}
An element $X\in \mathcal{P}(\mathbb{Z}_{\geq0}^m)$ generates a monomial ideal which contains a unique minimal basis $B(X)$ (see e.g.~\cite{CLS92}). In general,   $\Vert(X)\subset B(X)$ and this inclusion may be strict. Consider the set $X=\{A_1=(1,4),A_2=(2,3),A_3=(3,3),A_4=(4,1)\} \subseteq \mathbb{Z}_{\geq 0}^2$. The Newton polygon $\mathcal{N}(X)$ can be visualized as in Figure \ref{Figure_NP} and $\mathrm{Vert}(X) = \{A_1, A_4\}$ which is a strict subset of $B(X)=\{A_1,A_2,A_4\}.$

\begin{figure}[!htb]
    \centering
    \includegraphics{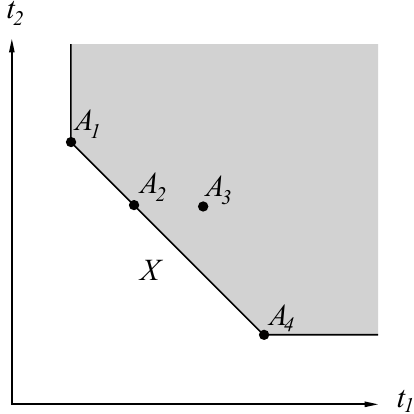}
    \caption{The Newton polygon of $X$. The vertex set of $X$ is $\{A_1, A_4\}$.}
\label{Figure_NP}
\end{figure}
 
 
\end{example}

We deduce from \cref{equal_Vert_is_equal_N} that the map $\mathrm{Vert}\colon\mathcal{P}(\mathbb{Z}^m_{\geq0})\longrightarrow \mathcal{P}(\mathbb{Z}^m_{\geq0})$ is a projection operator in the sense that $\mathrm{Vert}^2=\mathrm{Vert}$. 
\begin{definition}
We denote by $\mathbb{T}[[t_1,\ldots,t_m]]$ the image of the operator $\mathrm{Vert}$, and call its elements either \textbf{vertex sets} or \textbf{tropical formal power series}. For $S,T \in \mathbb{T}[[t_1,\ldots,t_m]]$, we define \[S\oplus T=\mathrm{Vert}(S\cup T) \quad \textrm{ and }\quad \ S\odot T=\mathrm{Vert}(S+T).\]
\end{definition}

\begin{corollary} \label{lemma:commutative semiring}
The set $(\mathbb{T}[[t_1,\ldots,t_m]],\oplus,\odot)$ is a commutative idempotent semiring, with the zero element $\emptyset$ and the unit element $\{(0,\ldots,0)\}$. 
\end{corollary}
\begin{proof}
The only things to check are associativity of $\oplus$, associativity of $\odot$ and the distributive property. The associativity of $\oplus$ and $\odot$ follows from the equalities 
\[S\oplus(T\oplus U)=\mathrm{Vert}(S\cup T\cup U)= (S\oplus T)\oplus U\]
and
\[S\odot(T\odot U)=\mathrm{Vert}(S+ T+ U)= (S\odot T)\odot U\] 
which are consequences of \cref{lemma_Properties_vert}. The distributivity follows from 
\[S\odot(T\oplus U)=\Vert((S+T)\cup U)=\Vert((S+T)\cup (S+U))=(S\odot T)\oplus (S\odot U).\qedhere\]
\end{proof}

\begin{corollary}
\label{Lemma_linearity}
The map $\Vert$ is a homomorphism of  semirings. In particular, for any finite family $\{X_i\}_i$ of elements $X_i\in\mathcal{P}(\mathbb{Z}_{\geq0}^m)$, we have $\Vert(\sum_iX_i) = \bigodot_{i\in I}\Vert(X_i)$, $\Vert(\bigcup_iX_i) = \bigoplus_{i\in I}\Vert(X_i)$ and $\Vert(nT)=\Vert(T)^{\odot n}$.
\end{corollary}
\begin{proof}
Follows directly from \cref{lemma_Properties_vert} and \cref{lemma:commutative semiring}.
\end{proof}

\section{The differential ring of power series and the support map}
\label{S_DRPS}
We consider the differential ring $R_m$ from Section~\ref{S_PDAG}, and the semirings $\mathcal{P}(\mathbb{Z}^m_{\geq0})$, $\mathbb{T}[[t_1,\ldots,t_m]]$ from Section~\ref{sec:the tropical semi-ring}. In this part we introduce the support and the tropicalization maps, which are related by the following commutative diagram \begin{equation*}
    \xymatrix{R_m\ar[r]^{\text{Supp}}\ar[dr]_{\text{trop}}&\mathcal{P}(\mathbb{Z}^m_{\geq0})\ar[d]^{\Vert}\\
    &\mathbb{T}[[t_1,\ldots,t_m]]\\}
\end{equation*}

If $J=(j_1, \ldots, j_m)$ is an element of $\mathbb{Z}^m_{\geq0}$, we will denote by $t^J$ the monomial $t_1^{j_1} \cdots t_m^{j_m}$.
An element of $R_m$ is of the form $\varphi = \sum_{J \in \mathbb{Z}^m_{\geq0}} a_J t^J$ with  $a_J \in K.$

\begin{definition}
 The \textbf{support} of $\varphi = \sum a_J t^J \in R_m$ is defined as 
 \[\mathrm{Supp}(\varphi) = \{J \in \mathbb{Z}^m_{\geq0} \ |\ a_J \neq 0\}.\]
For a fixed integer $n$, the map which sends $\varphi=(\varphi_1, \ldots, \varphi_n) \in R_m^n$ to $\mathrm{Supp}(\varphi)= (\mathrm{Supp}(\varphi_1), \ldots, \mathrm{Supp}(\varphi_n)) \in \mathcal{P}(\mathbb{Z}^m_{\geq0})^n$ will also be denoted by $\mathrm{Supp}$. The \textbf{set of supports} of a subset $T\subseteq R_m^n$ is its image under the map Supp:
\[\mathrm{Supp}(T) = \{\mathrm{Supp}(\varphi) \ | \ \varphi \in T\} \subseteq \mathcal{P}(\mathbb{Z}^m_{\geq0}) ^n \]
\end{definition}

\begin{definition}
The mapping that sends each series in $R_m$ to the vertex set of its support  is called the \textbf{tropicalization} map
\[\begin{array}{cccc}
\mathrm{trop} \colon &R_m & \to &\mathbb{T}[[t_1,\ldots,t_m]]  \\
& \varphi & \mapsto & \Vert(\mathrm{Supp}(\varphi))
\end{array}\]
\end{definition}


\begin{lemma}
\label{lem:trop_morphism_of_semi_ring}
The tropicalization map is a non-degenerate valuation in the sense of \cite[Definition 2.5.1]{Giansiracusa-Giansiracusa}. This is, it satisfies
\begin{enumerate}
    \item $\trop(0)=\emptyset$, $\trop(\pm1)=\{(0,\ldots,0)\}$,
    \item $\trop(\varphi\cdot\psi)=\trop(\varphi)\odot\trop(\psi)$,
    \item $\trop(\varphi + \psi) \oplus \trop(\varphi) \oplus \trop(\psi) = \trop(\varphi) \oplus \trop(\psi)$,
    \item $\trop(\varphi)=\emptyset$ implies that $\varphi=0$.
\end{enumerate}
\end{lemma}
\begin{proof}
The first point is clear. For the second point, note that the Newton polygon has the well-known homomorphism-type property \[\mathcal{N}(\Supp(\varphi \cdot \psi))= \mathcal{N}(\Supp(\varphi))+\mathcal{N}(\Supp(\psi))=\mathcal{N}(\Supp(\varphi)+\Supp(\psi)).\]
Hence, the vertices of the left hand side coincide with the vertices of the right hand side. This gives $\trop(\varphi\cdot\psi) = \Vert(\mathcal{N}(\Supp(\varphi) + \Supp(\psi)))$. That this is equal to $\trop(\varphi)\odot\trop(\psi)$ follows from \cref{lemma_Properties_vert}. The third point follows from the observation that $\Supp(\varphi + \psi) \subseteq \Supp(\varphi) \cup \Supp(\psi)$ and \cref{Lemma_linearity}. The last point follows from the fact that the empty set is the only set with empty Newton polygon.
\end{proof}

\begin{definition}
For $J = (j_1, \ldots, j_m) \in \mathbb{Z}^m_{\geq 0}$, we define the \textbf{tropical derivative operator} $\Theta_{\trop}(J)\colon\mathcal{P}(\mathbb{Z}^m_{\geq0})\to\mathcal{P}(\mathbb{Z}^m_{\geq0})$ as
\[\Theta_{\trop}(J)T:=\left\{ (t_1-j_1, \ldots, t_m -j_m) \ \left| \ 
\begin{array}{cccc} (t_1, \ldots, t_m) \in T, \\
  t_i-j_i \ge 0 \text{ for all }i
 \end{array}
 \right. \right\}\]
\end{definition}

For example, if $T$ is the grey part in Figure~\ref{F_DS2} left and $J=(1,2)$, then informally $\Theta_{\trop}(J)T$ is a translation of $T$ by the vector $-J$ and then keeping only the non-negative part. 
It is represented by the grey part in Figure~\ref{F_DS2} right.

\begin{figure}[!htb]
    \centering
    \includegraphics{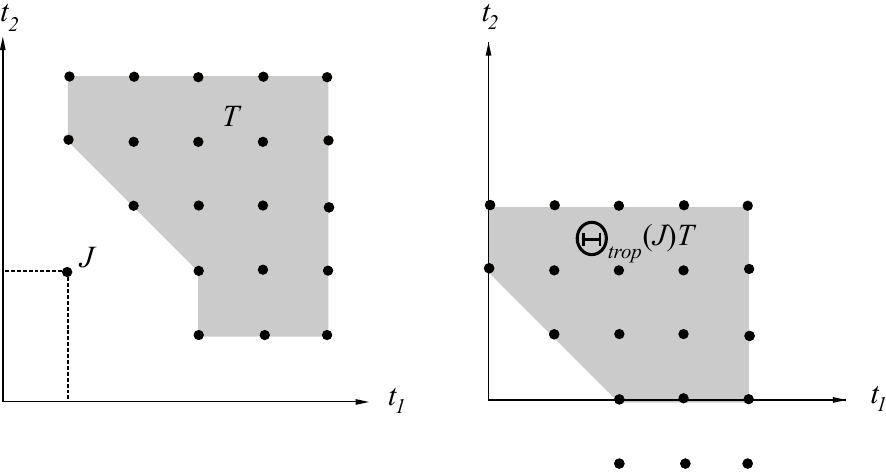}
    \caption{The operator $\Theta_{\trop}(J)$ for $J=(1,2)$ 
    applied to  $T$.}
   \label{F_DS2}
\end{figure}




\medskip
Since $K$ is of characteristic zero, for all $\varphi \in R_m$ and $J \in \mathbb{Z}^m_{\geq0}$, we have
\begin{equation}
\label{dif_comm}
\mathrm{Supp}(\Theta(J)\varphi) = \Theta_{\trop}(J)\mathrm{Supp}(\varphi)
\end{equation}

Consider a differential monomial $E_M$ as in \eqref{differential_monomial} and $S=(S_1,\ldots,S_n)\in\mathcal{P}(\mathbb{Z}_{\geq0}^m)^n$. We can now define the evaluation of  $E_M$ at $S$ as 
\begin{equation}
\label{expression_evaluation}
    E_M(S) = \sum_{\substack{1\leq i\leq n\\ ||J||_\infty\leq r}} M_{i,J}\Theta_{\trop}(J)S_i\in\mathcal{P}(\mathbb{Z}_{\geq0}^m).
\end{equation}

\begin{lemma}
\label{lem:monomial equalitymod}
Given $\varphi=(\varphi_1,\ldots,\varphi_n) \in R_m^n$ and a differential monomial $E_M$, we have
$\trop(E_M(\varphi))=\Vert(E_M(\Supp(\varphi)))$
\end{lemma}
\begin{proof}
By applying $\Vert$ to equation \eqref{dif_comm}, we have 
\begin{equation} \label{eq-help1}
\trop(\Theta(J)\varphi_i) = \Vert(\Theta_{\trop}(J)\mathrm{Supp}(\varphi_i)).
\end{equation}
Using the multiplicativity of trop, equation \eqref{eq-help1} and \cref{Lemma_linearity}, we obtain \begin{align*}
    \trop(E_M(\varphi))&=\bigodot_{i,J}\trop(\Theta(J)\varphi_i)^{\odot M_{i,J}}\\ &=\bigodot_{i,J}\Vert(\Theta_{\trop}(J)\mathrm{Supp}(\varphi_i))^{\odot M_{i,J}}\\ &=\Vert(E_M(\text{Supp}(\varphi))).\qedhere
\end{align*}
\end{proof}

\begin{remark}
If $P=\sum_M\alpha_ME_M \in R_{m,n}$ and $\varphi=(\varphi_1,\ldots,\varphi_n)\in R_m^n$, then we can consider the upper support $US(P,\varphi)$ of $P$ at $\varphi$ as $$US(P,\varphi)=\bigcup_M\left(\Supp(\alpha_M)+\Supp(E_M(\varphi)) \right)\in\mathcal{P}(\mathbb{Z}_{\geq0}^m).$$

We now compute the vertex set of $US(P,\varphi)$ by applying the operation $\Vert$ and \cref{Lemma_linearity} to the above expression to find
\begin{equation*}\begin{aligned} 
\Vert\Bigl(&US(P,\varphi) \Bigr)=\bigoplus_M\trop(\alpha_M)\odot\trop(E_M(\varphi))\\ &=\bigoplus_M\trop(\alpha_M)\odot\Vert(E_M(\text{Supp}(\varphi))),
\end{aligned} \end{equation*} 
since $\trop(E_M(\varphi))=\Vert(E_M(\Supp(\varphi)))$ by Lemma \ref{lem:monomial equalitymod}. This motivates the definition of tropical differential polynomials in the next section.
\end{remark}

\section{Tropical differential polynomials}
\label{S_TDP}
In this section we define the set of tropical differential polynomials $\mathbb{T}_{m,n}$ and the corresponding tropicalization morphism $\trop\colon R_{m,n}\to\mathbb{T}_{m,n}$. 
Let us remark that in the case of $m=1$ the definitions and properties presented here coincide with the corresponding ones in~\cite{AGT16}. Moreover, later in Section~\ref{Section_ExLim} we illustrate in \cref{example:NewtonPolygon} the reason for the particular definitions given here.

\begin{definition}
\label{def:definition ValS}
For a set $S \in \mathcal{P}(\mathbb{Z}^m_{\geq0})$ and a multi-index $J \in \mathbb{Z}^m_{\geq0}$ we define
\[
\mathrm{Val}_J(S) = \Vert(\Theta_{\trop}(J)S).
\]
\end{definition}
Note that for $\varphi \in R_m$ and any multi-index $J$ this means that
\[\Val_J(\Supp(\varphi)) = \trop(\Theta(J)\varphi).\]
In particular, $\mathrm{Val}_J(S) = \emptyset$ if and only if $\Theta(J)\varphi = 0$. It follows from \cref{Lemma_linearity} that \[\Vert(E_M(S))=\bigodot_{\substack{1\leq i\leq n\\ ||J||_\infty\leq r}} \mathrm{Val}_J(S_i)^{\odot M_{i,J}}.\]

\begin{definition}
A \textbf{tropical differential monomial} in the variables $x_1, \ldots, x_n$ of order less or equal to $r$ is an expression of the form 
 \[\epsilon_M=\bigodot_{\substack{1\leq i\leq n\\ ||J||_\infty\leq r}}x_{i,J}^{\odot M_{i,J}}\]
where $M=(M_{i,J}) \in (\mathbb{Z}_{\ge0})^{n\times(r+1)^m}$. 
\end{definition}

A tropical differential monomial $\epsilon_M$ induces an evaluation map from $\mathcal{P}(\mathbb{Z}^m_{\geq0})^n$ to $\mathbb{T}[[t_1,\ldots,t_m]]$ by
\[\epsilon_M(S_1,\ldots,S_n) =\Vert(E_M(S ))= \bigodot_{i,J} \mathrm{Val}_J(S_i)^{\odot M_{i,J}}\]
where $\mathrm{Val}_J(S_i)$ is given in \cref{def:definition ValS} and $E_M(S)$ as in \eqref{expression_evaluation}. Let us recall that, by \cref{lemma:commutative semiring}, we can also write
\[\epsilon_M(S_1, \ldots, S_n)= \mathrm{Vert}\biggl(\sum_{i,J}\mathrm{Val}_J(S_i)^{\odot M_{i,J}}
\biggr).\]

\begin{definition}
\label{def:tropical diff poly}
A \textbf{tropical differential polynomial} in the variables $x_1, \ldots, x_n$ of order less or equal to $r$ is an expression of the form 
\[p = p(x_1, \ldots, x_n) = \bigoplus_{M \in \Delta } a_M \odot \epsilon_M \]
where $a_M \in \mathbb{T}[[t_1,\ldots,t_m]], a_M \neq \emptyset$ and $\Delta$ is a finite subset of $(\mathbb{Z}_{\ge0})^{n\times(r+1)^m}$. We denote by $\mathbb{T}_{m,n}=\mathbb{T}[[t_1,\ldots,t_m]]\{x_1,\ldots,x_n\}$
the set of tropical differential polynomials.
\end{definition}

A tropical differential polynomial $p$ as in Definition \ref{def:tropical diff poly} induces a mapping from $\mathcal{P}(\mathbb{Z}^m_{\geq0})^n$ to $\mathbb{T}[[t_1,\ldots,t_m]]$ by
\[p(S) = \bigoplus_{M \in \Delta } a_M \odot \epsilon_M(S) = \mathrm{Vert}\Bigl( \bigcup_{M \in \Delta} (a_M + \epsilon_M(S))\Bigr) \]
The second equality follows again from \cref{lemma:commutative semiring}.
A differential polynomial $P \in R_{m,n}$ of order at most $r$ is of the form \[P= \sum_{M\in \Delta} \alpha_M E_M\]
where $\Delta$ is a finite subset of $(\mathbb{Z}_{\ge0})^{n\times(r+1)^m}$, $\alpha_M \in K[[t_1, \ldots, t_m]]$ and $E_M$ is a differential monomial as in \eqref{differential_monomial}.
Then the \textbf{tropicalization} of $P$ is defined as
\[\mathrm{trop}(P) = \bigoplus_{M \in \Delta }\trop(\alpha_M) \odot \epsilon_M\in \mathbb{T}_{m,n}\]
where $\epsilon_M$ is the tropical differential monomial corresponding to $E_M$. 

\begin{definition}
Let $G \subseteq R_{m,n}$ be a differential ideal. 
Its \textbf{tropicalization} $\mathrm{trop}(G)$ is the set of tropical differential polynomials $\{\mathrm{trop}(P) \ | \ P \in G\} \subseteq \mathbb{T}_{m,n}$. 
\end{definition}

\begin{lemma} \label{lem:monomial equality}
Given a differential monomial $E_M$ and $\varphi= (\varphi_1, \ldots, \varphi_n) \in K[[t_1, \ldots, t_m]]^n$, we have that
\[\trop(E_M(\varphi)) = \epsilon_M(\Supp(\varphi)). \]
\end{lemma}
\begin{proof}
Follows from notations and \cref{lem:monomial equalitymod}.
\end{proof}

Consider $P=\sum_M\alpha_ME_M \in R_{m,n}$ and $\varphi=(\varphi_1,\ldots,\varphi_n)\in R_m^n$. Set $p=\trop(P)$ and $S=\Supp(\varphi)$, then  $p(S)=\Vert(US(P,\varphi))$.

\medskip
The following tropical vanishing condition is a natural generalization of the case $m=1$, but now the evaluation $p(S)$ consists of a vertex set instead of a single minimum.

\begin{definition}
Let $p = \bigoplus_{M \in \Delta } a_M \odot \epsilon_M$ be a tropical differential polynomial. An $n$-tuple $S \in \mathcal{P}(\mathbb{Z}^m_{\geq0})^n$ is said to be a \textbf{solution} of $p$ if for every $J \in p(S)$ there exists $M_1, M_2 \in \Delta$ with $M_1 \neq M_2$ such that $J \in a_{M_1} \odot \epsilon_{M_1}(S)$ and $J \in a_{M_2} \odot \epsilon_{M_2}(S)$. Note that in the particular case of $p(S) = \emptyset$, $S$ is a solution of $p$.

\medskip
For a family of differential polynomials $H \subseteq \mathbb{T}_{m,n}$, $S$ is called a \textbf{solution} of $H$ if and only if $S$ is a solution of every tropical polynomial in $H$. The set of solutions of $H$ will be denoted by $\mathrm{Sol}(H)$.
\end{definition}

\begin{proposition}\label{easy-direction}
Let $G$ be a differential ideal in the ring of differential polynomials $R_{m,n}$.
If $\varphi \in \Sol(G)$, then $\Supp(\varphi) \in \Sol(\trop(G))$.
\end{proposition}

\begin{proof}
Let $\varphi$ be a solution of $G$ and $S = \Supp(\varphi)$. Let $P= \sum_{M \in \Delta} \alpha_M  E_M \in G$ 
and $p = \trop(P) = \bigoplus_{M \in \Delta}a_M \odot \epsilon_M,$ where $a_M = \trop(\alpha_M)$. We need to show that $S$ is a solution of $p$. 
Let $J \in p(S)$ be arbitrary. 
By the definition of $\oplus$, there is an index $M_1$ such that \[J \in a_{M_1} \odot \epsilon_{M_1}(S).\]
Hence, by \cref{lem:monomial equality}, and multiplicative property of trop \cref{lem:trop_morphism_of_semi_ring}
\[J \in \Vert(\Supp(\alpha_{M_1} E_{M_1}(\varphi))).\] 
Since $P(\varphi) = 0$, there is another index $M_2 \neq M_1$ such that \[J \in \Supp(\alpha_{M_2} E_{M_2}(\varphi)),\] because otherwise there would not be cancellation. 
Since $J$ is a vertex of $p(S)$, it follows that $J$ is a vertex of every subset of $\mathcal{N}(p(S))$ containing $J$ and in particular of $\mathcal{N}(\mathrm{Supp}(\alpha_{M_2} E_{M_2}(\varphi)))$. 
Therefore, 
\[J \in a_{M_2} \odot \epsilon_{M_2}(S)\] 
and because $J$ and $P$ were chosen arbitrary, $S$ is a solution of $G$.
\end{proof}

\section{The Fundamental Theorem}
\label{S_TFT}
Let $G \subset R_{m,n}$ be a differential ideal. Then \cref{easy-direction} implies that $\Supp(\Sol(G)) \subseteq \Sol(\trop(G))$. The main result of this paper is to show that the reverse inclusion holds as well if the base field $K$ is uncountable.
\begin{thm}[Fundamental Theorem]\label{fundamental theorem}
Let $K$ be an uncountable, algebraically closed field of characteristic zero. Let $G$ be a differential ideal in the ring $R_{m,n}$. Then 
\[ \Supp(\Sol(G)) = \Sol(\trop(G)). \]
\end{thm}


The proof of the Fundamental Theorem will take the rest of the section and is split into several parts. First let us introduce some notations. If $J=(j_1, \ldots, j_m)$ is an element of $\mathbb{Z}^m_{\geq0}$, we define by $J!$ the component-wise product $j_1! \cdots j_m !$.
The bijection between $K^{\mathbb{Z}^m_{\geq0}}$ and $R_m$ given by
\[\begin{array}{cccc}
\psi \colon & K^{\mathbb{Z}^m_{\geq0}} & \to & R_m \\
& \underline{a}=(a_J)_{J \in \mathbb{Z}^m_{\geq0}} &  \mapsto & \displaystyle \sum_{J \in \mathbb{Z}^m_{\geq0}} \frac{1}{J!}a_J t^J
\end{array}\]
allows us to identify points of $R_m$ with points of $K^{\mathbb{Z}^m_{\geq0}}$.
Moreover, if $I \in \mathbb{Z}^m_{\geq0}$, the mapping $\psi$ has the following property:
\[\Theta(I)\psi(\underline{a}) = \sum_{J \in \mathbb{Z}^m_{\geq0}} \frac{1}{J!}a_{I+J} t^J\]
which implies
\[\underline{a} = (\Theta(I)\psi(\underline{a})|_{t=0})_{I \in \mathbb{Z}^m_{\geq0}}.\]

Fix for the rest of the section a finite set of differential polynomials $\Sigma=\{P_1, \ldots, P_s\} \subseteq G$ such that $\Sigma$ has the same solution set as $G$ (this is possible by \cref{prop_Ritt_Raudenbush}). For all $\ell \in \{1, \ldots, s\}$ and $I \in \mathbb{Z}_{\geq 0}^m$ we define
\[ F_{\ell, I} = ( \Theta(I)P_\ell)|_{t_1 = \cdots = t_m = 0} \, \in K\big[x_{i, J} : 1 \leq i \leq n, J \in \mathbb{Z}_{\geq 0}^m\big] \]
and
\[ A_{\infty} = \{(a_{i,J}) \in K^{n \times (\mathbb{Z}_{\geq 0}^m)} : F_{\ell, I}(a_{i,J}) = 0 \textrm{ for all } 1 \leq \ell\leq s, I \in \mathbb{Z}_{\geq 0}^m \}.\]
The set $A_{\infty}$ corresponds to the formal power series solutions of the differential system $\Sigma = 0$ as the following lemma shows.

\begin{lemma}
Let $\varphi \in K[[t_1, \ldots, t_m]]^n$ with $\varphi = (\varphi_1, \ldots, \varphi_n)$, where 
\[\varphi_i = \sum_{J \in \mathbb{Z}_{\geq 0}^m} \frac{a_{i,J}}{J!} t^J.\]
Then $\varphi$ is a solution of $\Sigma = 0$ if and only if $(a_{i,J}) \in A_\infty$.
\end{lemma}

\begin{proof}
This statement follows from formula
\[ P_\ell(\varphi_1,\ldots, \varphi_n) = \sum_{I \in \mathbb{Z}_{\geq 0}^m} \frac{F_{\ell,I}((a_{i,J})_{i,J})}{I!}t^I,\]
which is often called Taylor formula. See~\cite{Seidenberg58} for more details.
\end{proof}

For any $S = (S_1, \ldots, S_n) \in \mathcal{P}(\mathbb{Z}_{\geq 0}^m)^n$ we define 
\[ A_{\infty, S} = \{ (a_{i,J}) \in A_\infty : a_{i,J} = 0 \textrm{ if and only if } J \notin S_i\}.\]
This set corresponds to power series solutions of the system $\Sigma = 0$ which have support exactly $S$. In particular, $S \in \Supp(\Sol(G))$ if and only if $A_{\infty, S} \neq \emptyset$. 

\medskip
The sets $A_{\infty}$ and $A_{\infty, S}$ refer to infinitely many coefficients. We want to work with a finite approximation of these sets. For this purpose, we make the following definitions. For each integer $k \geq 0$, choose $N_k \geq 0$ minimal such that for every $\ell \in \{1, \ldots, s\}$ and $||I||_\infty \leq k$ it holds that
\[ F_{\ell, I} \in K[x_{i, J} : 1 \leq i \leq n, ||J||_\infty \leq N_k].\]
Note that for $k_1 \leq k_2$ it follows that $N_{k_1} \leq N_{k_2}.$
Then we define 
\begin{align*}
    A_k = \{(a_{i,J}) \in &K^{n \times \{1, \ldots, N_k\}^m} : F_{\ell, I}(a_{i,J}) = 0 \\
    &\quad\quad\textrm{ for all } 1 \leq \ell\leq s, ||I||_\infty \leq k \}
\end{align*}

and 
\[ A_{k,S} = \{ (a_{i,J}) \in A_k : a_{i,J} = 0 \textrm{ if and only if } J \notin S_i\}.\]

\begin{proposition}\label{finite-approximation}
Let $S \in \mathcal{P}(\mathbb{Z}_{\geq 0}^m)^n$. If $A_{\infty, S} = \emptyset$, then there exists $k \geq 0$ such that $A_{k, S} = \emptyset$. 
\end{proposition}
\begin{proof}
Assume that $A_{k,S} \neq \emptyset$ for every $k \geq 0$; we show that this implies $A_{\infty, S} \neq \emptyset$. We follow the strategy of the proof of \cite[Theorem 2.10]{DenefLipshitz}: first we use the ultrapower construction to construct a larger field $\mathbb{K}$ over which a power series solution with support $S$ exists, and then we show that this implies the existence of a solution with the same support and with coefficients in $K$. For more information on ultrafilters and ultraproducts, the reader may consult~\cite{BDLD79}.

For each integer $k \geq 0$, choose an element $(a^{(k)}_{i,J})_{1 \leq i \leq n, ||J||_{\infty} \leq N_k} \in A_{k,S}$. Fix a non-principal ultrafilter $\mathcal{U}$ on the natural numbers $\mathbb{N}$ and consider the ultrapower $\mathbb{K}$ of $K$ along $\mathcal{U}$. In other words, $\mathbb{K} = (\prod_{r \in \mathbb{N}} K) / \sim$ where $x \sim y$ for $x = (x_r)_{r \in \mathbb{N}}$ and $y = (y_r)_{r \in \mathbb{N}}$ if and only if the set $\{r \in \mathbb{N} : x_r = y_r\}$ is in $\mathcal{U}$. We will denote the equivalence class of a sequence $(x_r)$ by $[(x_r)]$. We consider $\mathbb{K}$ as a $K$-algebra via the diagonal map $K \to \mathbb{K}$. Now for each $i$ and $J$, we may define $a_{i,J} \in \mathbb{K}$ as 
\[ a_{i,J} = [ (a_{i,J}^{(k)} : k \in \mathbb{N}) ] \]
where we set $a_{i,J}^{(k)} = 0$ for the finitely many values of $k$ with $||J||_{\infty} > N_k$. For all $\ell$ and $I$, we have that $F_{\ell, I}((a^{(k)}_{i,J})_{i,J}) = 0$ for $k$ large enough, and so $F_{\ell, I}((a_{i,J})_{i,J}) = 0$ in $\mathbb{K}$, because the set of $k$ such that $F_{\ell, I}((a^{(k)}_{i,J})_{i,J}) \neq 0$ is finite. Moreover, for $J \in S_i$ we have, by hypothesis, $a_{i,J}^{(k)} \neq 0$ for all sufficiently large $k$, so $a_{i,J} \neq 0$ in $\mathbb{K}$. On the other hand, for $J \notin S_i$ we have $a_{i,J}^{(k)} = 0$ for all $k$, so also $a_{i,J} = 0$. 
Now consider the ring 
\[R = K\left[\begin{array}{l}
             x_{i,J}: 1 \leq i \leq n, J \in \mathbb{Z}^m_{\geq 0}\\
             x_{i,J}^{-1} : 1 \leq i \leq n, J \in S_i
             \end{array}\right] / \left(\begin{array}{l}
             F_{\ell, I} : 1 \leq \ell \leq s,  I \in \mathbb{Z}_{\geq 0}^m\\
             x_{i,J}: 1 \leq i \leq n, J \notin S_i
              \end{array}\right)
             \]
The paragraph above shows that the map $R \to \mathbb{K}$ defined by sending $x_{i,J}$ to $a_{i,J}$ is a well-defined ring map. In particular, $R$ is not the zero ring. Let $\mathfrak{m}$ be a maximal ideal of $R$. We claim that the composition $K \to R \to R/\mathfrak{m}$ is an isomorphism. Indeed, $R/\mathfrak{m}$ is a field, and as a $K$-algebra it is countably generated, since $R$ is. Therefore, it is of countable dimension as $K$-vector space (it is generated as $K$-vector space by the products of some set of generators as a $K$-algebra). If $t \in R/\mathfrak{m}$ were transcendental over $K$, then by the theory of partial fraction decomposition, the elements $1/(t - \alpha)$ for $\alpha \in K$ would form an uncountable, $K$-linearly independent subset of $R/\mathfrak{m}$. This is not possible, so $R/\mathfrak{m}$ is algebraic over $K$. Since $K$ is algebraically closed, we conclude that $K = R/\mathfrak{m}$. Now let $b_{i,J} \in K$ be the image of $x_{i,J}$ in $R/\mathfrak{m}=K$. Then by construction, the set $(b_{i,J})$ satisfies the conditions $F_{\ell,I}((b_{i,J})) = 0$ for all $\ell$ and $I$, and $b_{i, J} = 0$ if and only if $J \notin S_i$. So $(b_{i,J})$ is an element of $A_{\infty, S}$, and in particular $A_{\infty, S} \neq \emptyset$. 
\end{proof}

\begin{proof}[Proof of \cref{fundamental theorem}] 
We now prove the remaining direction of the Fundamental Theorem by contraposition. Let $S = (S_1, \ldots, S_n)$ in $\mathcal{P}(\mathbb{Z}_{\geq 0}^m)^n$ be such that $A_{\infty, S} = \emptyset$, i.e. there is no power series solution of $\Sigma = 0$ in $K[[t_1, \ldots, t_m]]^n$ with $S$ as the support. Then by \cref{finite-approximation} there exists $k \geq 0$ such that $A_{k,S} = \emptyset$. Equivalently, 
\[ V\left(\begin{array}{l}
     F_{\ell, I} : 1 \leq \ell \leq s, ||I||_{\infty} \leq N_k \\
     x_{i,J} : 1 \leq i \leq n, J \notin S_i, ||J||_{\infty} \leq k
\end{array}\right) \subseteq V\Biggl(\prod_{\substack{1 \leq i \leq n\\ J \in S_i\\ ||J||_{\infty} \leq N_k}} x_{i,J} \Biggr).\]
By Hilbert's Nullstellensatz, there is an integer $M \geq 1$ such that
\[ E := \bigg(\prod_{\substack{1 \leq i \leq n\\ J \in S_i\\||J||_{\infty} \leq N_k}} x_{i,J}\bigg)^M \in  \left\langle\begin{array}{l}
     F_{\ell, I} : 1 \leq \ell \leq s, ||I||_{\infty} \leq N_k \\
     x_{i,J} : 1 \leq i \leq n, J \notin S_i, ||J||_{\infty} \leq k
\end{array}\right\rangle. \]
Therefore, there exist $G_{\ell, I}$ and $H_{i,J}$ in $K[x_{i,J} : 1 \leq i \leq n, ||J||_{\infty} \leq N_k]$ such that
\[ E = \sum_{\substack{1 \leq \ell \leq s\\ ||I||_{\infty} \leq k}} G_{\ell, I}F_{\ell, I} + \sum_{\substack{1 \leq i \leq n\\ J \notin S_i \\ ||J||_{\infty} \leq N_k}} H_{i,J} x_{i,J}. \]
Define the differential polynomial $P$ by 
\[P = \sum_{\substack{1 \leq \ell \leq s\\ ||I||_{\infty} \leq k}} G_{\ell, I}\Theta(I)(P_\ell).\]
Then $P$ is an element of the differential ideal generated by $P_1, \ldots, P_s$, so in particular $P \in G$. Since $F_{\ell, I} = \Theta(I)(P_\ell) |_{t = 0}$, there exist $h_i \in R_{m,n}$ such that
\[ P = E - \sum_{\substack{1 \leq i \leq n\\ J \notin S_i \\ ||J||_{\infty} \leq N_k}} H_{i,J} x_{i,J} \,+\, t_1 h_1 + \ldots + t_m h_m.\]
Notice that $E$ occurs as a monomial in $P$, since it cannot cancel with other terms in the sum above. By construction we have $\trop(E)(S) = \{(0,\ldots, 0)\}$. However, we have $(0, \ldots, 0) \notin \trop(H_{i,J} x_{i,J})(S)$ because $J \notin S_i$, and we have $(0, \ldots, 0) \notin \trop(t_i h_i)(S)$ because the factor $t_i$ forces the $i$th coefficient of each element of $\trop(t_i h_i)(S)$ to be at least $1$. Hence, the vertex $\{(0,\ldots, 0)\}$ in $\trop(P)(S)$ is attained exactly once, in the monomial $E$, and therefore, $S$ is not a solution of $\trop(P)$. Since $P \in G$, it follows that $S \notin \Sol(\trop(G))$, which proves the statement.
\end{proof}

\section{Examples and remarks on the Fundamental Theorem}
\label{Section_ExLim}
In this section we give some examples to illustrate the results obtained in the previous sections. Moreover, we show that some straight-forward generalizations of the Fundamental Theorem from~\cite{AGT16} and our version, Theorem~\ref{fundamental theorem}, do not hold. Also we give more directions for further developments.

\begin{example}
Let us consider the system of differential polynomials
\begin{align*}
\Sigma=\{&P_1=x_{1,(1,0)}^2 - 4\,x_{1,(0,0)}\,,\ 
P_2=x_{1,(1,1)}\,x_{2,(0,1)} - x_{1,(0,0)} + 1\,,\\ 
&P_3=x_{2,(2,0)} - x_{1,(1,0)}\}\,
\end{align*}
in $R_{2,2}$. 
By means of elimination methods in differential algebra 
such as the ones implemented in the MAPLE
\texttt{DifferentialAlgebra} package, it can be proven that
\begin{align*} \Sol(\Sigma)=\{
&\varphi_1(t_1,t_2)=2\,c_0\,t_1+c_0^2+\sqrt{2}\,c_0\,t_2+t_1^2+\sqrt{2}\,t_1\,t_2+\frac{1}{2}\,t_2^2,\\
&\varphi_2(t_1,t_2)=c_2\,t_1+c_1+\frac{1}{2}\,\sqrt{2}\,(c_0^2-1)\,t_2+c_0\,t_1^2 \\
 & \quad +\sqrt{2}\,c_0\,t_1\,t_2 +\frac{1}{2}\,c_0\,t_2^2 \\
& \quad +\frac{1}{3}\,t_1^3+\frac{1}{2}\,\sqrt{2}\,t_1^2\,t_2+\frac{1}{2}\,t_1\,t_2^2+\frac{1}{12}\,\sqrt{2}\,t_2^3 \},
\end{align*}
where~$c_0,c_1,c_2 \in K$ are arbitrary constants. 
By setting $c_0=c_2=0,c_1 \neq 0$, we obtain for example that \[(\{(2,0),(1,1),(0,2)\},\{(0,0),(0,1),(3,0),(2,1),(1,1),(0,3)\})\]
is in $\Supp(\Sol(\Sigma))$.

\medskip
Now we illustrate that by our results necessary conditions and relations on the support can be found.
Let $(S_1,S_2) \in \mathcal{P}(\mathbb{Z}_{\ge 0}^2)^2$ be a solution of $\trop([\Sigma])$. Let us first consider 
\[\trop(P_1)(S_1,S_2)=\Vert(2 \cdot \Theta_{\trop}(1,0)S_1 \, \cup \, S_1).\]
If we assume that $(0,0) \in S_1$, then $(0,0)$ is a vertex of $S_1$. By the definition of a solution of a tropical differential polynomial, $(0,0)$ must be a vertex of the term $2\cdot \Theta_{\trop}(1,0)S_1$ as well, so we then know that $(1,0) \in S_1$. Conversely, if $(1,0) \in S_1$, then $(0,0) \in S_1$ follows. This is what we expect since the corresponding monomials in $\varphi_1$ vanish if and only if $c_0=0$.\\
Now consider 
\begin{equation*}
    \begin{aligned}
    &\trop(\Theta(1,0)P_1)(S_1,S_2)=\\
    &\Vert(\Theta_{\trop}(1,0)S_1+\Theta_{\trop}(2,0)S_1\,\cup\,\Theta_{\trop}(1,0)S_1).
    \end{aligned}
\end{equation*}
If we assume that $(0,0)$ is not a vertex of this expression, which implies that $(1,0) \notin S_1$, and $(k,0)$ is a vertex for some $k \ge 1$, then we obtain from the two tropical differential monomials that necessarily $(k,0)=(2k-1,0)$. This is fulfilled only for $k=1$ and hence, $(2,0) \in S_1$.

\end{example}

\medskip
A natural way for defining $\odot$ and $\oplus$ in Section~\ref{sec:the tropical semi-ring} would be to simply take the Newton polygon and not take its vertex set, as we do. If we do this, then some intermediate results and in particular Proposition~\ref{easy-direction}, do not hold anymore as the following example shows.
\begin{example} \label{example:NewtonPolygon}
Let $\{e_1,\ldots,e_4\}$ be the standard basis for $\mathbb{Z}^4_{\geq0}$. We consider the differential ideal in $R_{4,1}=K[[t_1,\ldots,t_4]]\{x\}$ generated by
\[ P = x_{e_3}x_{e_4} + (-t_1^2 + t_2^2) x_{e_1+e_3} = \frac{\partial x}{\partial t_3}\cdot \frac{\partial x}{\partial t_4} + (-t_1^2 + t_2^2) \frac{\partial^2 x}{\partial t_1 \partial t_3} \] and the solution $\varphi = (t_1 + t_2)t_3 + (t_1 - t_2)t_4$.
Then 
\[\Supp(\varphi)=\{e_1+e_3, e_2+e_3, e_1+e_4, e_2+e_4\}.\]
On the other hand, for $S \in \mathcal{P}(\mathbb{Z}^4_{\geq0})$ we obtain
\begin{align*}
\trop(P)(S)=\Vert(&\Vert(\Theta_{\trop}(e_3)S+\Theta_{\trop}(e_4)S) \\ &\cup \mathrm{Vert}(2e_1+\Theta_{\trop}(e_1+e_3)S) \\ &\cup \mathrm{Vert}(2e_2+\Theta_{\trop}(e_1+e_3)S).
\end{align*}

If we set $S=\Supp(\varphi)$, we obtain 
\begin{align*}
\trop(P)(S)=\Vert(&\Vert(\{2e_1,e_1+e_2,2e_2\})\cup \{2e_1\} \, \cup \{2e_2\}).
\end{align*}
Since 
\[\Vert(\{2e_1,e_1+e_2,2e_2\})=\{2e_1,2e_2\},\]
every $J \in \trop(P)(S)$, namely $2e_1$ and $2e_2$, occurs in three monomials in $\trop(P)(S)$ and $S$ is indeed in $\Sol(\trop(P))$.
Note that in the Newton polygon the point $e_1+e_2$, which is not a vertex, comes from only one monomial in $\trop(P)(S)$. Therefore, it is necessary to consider the vertices instead of the whole Newton polygon such that for instance Proposition~\ref{easy-direction} holds.
\end{example}

\begin{remark}\label{example:CountableField}
The Fundamental Theorem for systems of partial differential equations over a countable field such as $\overline{\mathbb{Q}}$ does in general not hold anymore by the following reasoning. 
According to~\cite[Corollary 4.7]{DenefLipshitz}, there is a system of partial differential equations $G$ over $\mathbb{Q}$ having a solution in $\mathbb{C}[[t_1, \ldots, t_m]]$ but no solution in $\overline{\mathbb{Q}}[[t_1, \ldots, t_m]]$. Taking $K = \overline{\mathbb{Q}}$ as base field, we have $\Sol(\trop(G)) \neq \emptyset$ because $\Sol(\trop(G))=\Supp(\Sol(G))$ is non-empty in $\mathbb{C}$, but $\Supp(\Sol(G)) = \emptyset$.
\end{remark}

In this  paper we focus on formal power series solutions. A natural extension would be to consider formal Puiseux series instead. The following example shows that with the natural extension of our definitions to Puiseux series, the fundamental theorem does not hold, even for $m = n = 1$. 

\begin{example} \label{example:Puiseux}
Let us consider $R_{1,1}=K[t]\{x\}$ and the differential ideal generated by the differential polynomial \[P=2tx_{(1)}-x_{(0)}=2t \cdot \frac{\partial x}{\partial t}-x.\]
There is no non-zero formal power series solution $\varphi$ of $P=0$, but $\varphi=ct^{1/2}$ is for any $c \in K$ a solution. In fact, $\{\varphi\}$ is the set of all formal Puiseux series solutions.

\medskip
On the other hand, let $S \in \mathcal{P}(\mathbb{Z}_{\ge0})$. Then every point $J$ in
\begin{align*}
\trop(P)(S)&=\Vert(\Vert(\{1\}+(\Theta_{\trop}(1)S) \cup \Vert(S))
\end{align*}
occurs in both monomials except if $0 \in J$. Hence, for every $S \in \Sol(\trop(P))$ we know that $0 \notin S$. 
For every $I \ge 0$ we have that \[\Theta(I)P=2tx_{(I+1)}+(2I-1)x_{(I)} \in G\] and 
\begin{align*}
\trop(\Theta(I)P)(S)&=\Vert(\Vert(\{(1)\}+(\Theta_{\trop}((I+1)S) \cup \Vert(\Theta_{\trop}(I)S)).
\end{align*}
Similarly to above, every $J \in \trop(\Theta(I)P)(S)$ occurs in both monomials except if $I \in J$. Therefore, $I \notin S$ and so the only $S \in \mathcal{P}(\mathbb{Z}_{\ge0})$ with $S \in \Sol(\trop([P]))$ is $S = \emptyset$. Hence, $\Sol(\trop([P]))=\{\emptyset\}=\Supp(\Sol([P])).$

\medskip
Now we want to consider formal Puiseux series solutions instead of formal power series solutions. First,  $\{\Supp(\varphi)\}=\{\emptyset,\{1/2\}\}$.
Now let us set for $S \in \mathbb{Q}^m$ and $J=(j_1,\ldots,j_m) \in \mathbb{Z}_{\ge 0}^m$, the set $\Theta_{\trop}(J)S$ defined as
\[\left\{ (s_1-j_1, \ldots, s_m -j_m) \ \left| \ 
\begin{array}{cccc} (s_1, \ldots, s_m) \in S, \\
 \forall 1 \le i \le m, \ s_i<0 \text{ or } s_i-j_i \notin \mathbb{Z}_{<0} 
 \end{array}
 \right. \right\}\]
This is the natural definition, since only in the case when the exponent of a monomial is a non-negative integer, the derivative can be equal to zero. We have that $\Theta_{\trop}(J)(\Supp(\psi)) = \Supp(\Theta(J)\psi)$ for all Puiseux series $\psi$. For $\Val_J$ and the operations $\odot$ and $\oplus$ the definitions remain unchanged.

\medskip
Let $Q \in [P]$. Then \[Q=\sum_{k \in \mathcal{I}}Q_k \cdot \Theta(I_k)P\] for some index-set $\mathcal{I}$ and $Q_k \in R_{m,n}$. 
For every $I_k$ we know that $\Supp(\varphi)=\{(1/2)\} \in \Sol(\trop(\Theta(I_k)P))$. Let $\alpha \in \mathbb{Q} \cap (0,1)$. Then for every $J \in \trop(\Theta(I_k)P) \in \mathbb{Z}_{\ge 0}$ we have that $\Theta_{\trop}(J)\{(1/2)\}=\Theta_{\trop}(J)\{(\alpha)\}+\{(1/2-\alpha)\}$. 
Thus, $\{\alpha\} \in \Sol(\trop(\Theta(I_k)P))$. 
Since 
\[\trop(G_k \cdot \Theta(I_k)P)=\trop(G_k)\odot \trop(\Theta(I_k)P),\]
the solvability remains by multiplication with $G_k$. Therefore, $\{\alpha\} \in \Sol(\trop(G_k \cdot \Theta(I_k)P))$ and consequently, $\{\alpha\} \in \Sol(\trop([P]))$. However, $\{\alpha\} \notin \Supp(\Sol([P]))$ for $\alpha \neq 1/2$.

\medskip
We remark that $P$ is an ordinary differential polynomial and by similar computations as here, the straight-forward generalization from formal power series to formal Puiseux series fails for the Fundamental Theorem in~\cite{AGT16} as well.
\end{example}

We conclude this section by emphasizing that the Fundamental Theorem may help to find necessary conditions on the support of solutions of systems of partial differential equations, but in general it cannot be completely algorithmic. In fact, according to~\cite[Theorem 4.11]{DenefLipshitz}, already determining the existence of a formal power series solution of a linear system with formal power series coefficients is in general undecidable.

\subsection*{Acknowledgements}
This work was started during the Tropical Differential Algebra workshop, which took place on December 2019 at Queen Mary University of London. We thank the organizers and participants for valuable discussions and initiating this collaboration. In particular, we want to thank Fuensanta Aroca, Alex Fink, Jeffrey Giansiracusa and Dima Grigoriev for their helpful comments during this week.

\appendix

\bibliography{tpdag}
\bibliographystyle{plain}

\noindent 
\textsc{Research Institute for Symbolic Computation (RISC)\\ 
Johannes Kepler University Linz, Austria.}
\par\noindent
Web: \url{risc.jku.at/m/sebastian-falkensteiner}
\\

\noindent
\textsc{Centro de Investigaci\'on en Matem\'aticas, A.C. (CIMAT)\\
Jalisco S/N, Col. Valenciana CP. 36023 Guanajuato, Gto, M\'exico.
}
\\
e-mail: \verb+cristhian.garay@cimat.mx+
\par\medskip\noindent

\noindent
\textsc{Institut de recherche math\'ematique de Rennes, UMR 6625 du CNRS\\ Universit\'e de Rennes 1, Campus de Beaulieu\\ 35042 Rennes cedex (France)}
\\
e-mail: \verb+mercedes.haiech@univ-rennes1.fr+
\par\medskip\noindent

\noindent
\textsc{Bernoulli Institute, University of Groningen, The Netherlands}\\
Web: \url{https://www.rug.nl/staff/m.p.noordman/}

\par\medskip\noindent
\textsc{School of Mathematical Sciences, Queen Mary University of London}\\ e-mail: \verb+z.toghani@qmul.ac.uk+

\par\medskip\noindent
\textsc{Univ. Lille, CNRS, Centrale Lille, Inria, UMR 9189 - CRIStAL - Centre de Recherche en Informatique Signalet Automatique de Lille, F-59000 Lille, France}\\
Web: \url{pro.univ-lille.fr/francois-boulier}

\end{document}